\documentclass{amsart}

\newtheorem{theorem}{Theorem}[section]
\newtheorem{lemma}[theorem]{Lemma}
\newtheorem{proposition}[theorem]{Proposition}
\theoremstyle{definition}
\newtheorem{definition}[theorem]{Definition}
\newtheorem{example}[theorem]{Example}

\theoremstyle{remark}
\newtheorem{remark}[theorem]{Remark}

\numberwithin{equation}{section}

\begin{document}

\title[The Radon-Nikod$\acute{Y}$m property of $\mathbb{L}$-Banach spaces]{The Radon-Nikod$\acute{Y}$m property of $\mathbb{L}$-Banach spaces and the dual representation theorem of $\mathbb{L}$-Bochner function spaces}
\author{Xia Zhang}
\address{School of Mathematical Sciences, TianGong University, No. 399 Binshui Road Xiqing District, Tianjin 300387, China}
\email{zhangxia@tiangong.edu.cn}
\thanks{The authors were supported by the National Natural Science Foundation of China Grant \#12171361 and the Humanities and Social Science Foundation of Ministry of Education Grant \#20YJC790174.}

\author{Xiangle Yan}
\address{School of Mathematical Sciences, TianGong University, No. 399 Binshui Road Xiqing District, Tianjin 300387, China}
\email{yanxiangle0610@163.com}

\author{Ming Liu}
\address{School of Mathematical Sciences, TianGong University, No. 399 Binshui Road Xiqing District, Tianjin 300387, China}
\email{liuming@tiangong.edu.cn}

\subjclass[2020]{Primary 46B22, 46B10; Secondary 46E30}

\date{September 8, 2024 and, in revised form, xx xx, 2024.}


\keywords{$\mathbb{L}$-Banach spaces, $\mathbb{L}$-$\mu$-measurable functions, $\mathbb{L}$-Bochner integrable functions, Radon-Nikod$\acute{y}$m property, dual representation theorem}

\begin{abstract}
In this paper, we first introduce $\mathbb{L}$-$\mu$-measurable functions and $\mathbb{L}$-Bochner integrable functions on a finite measure space $(S,\mathcal{F},\mu),$ and give an $\mathbb{L}$-valued analogue of the canonical $L^{p}(\Omega,\mathcal{F},\mu).$ Then we investigate the completeness of such an $\mathbb{L}$-valued analogue and propose the Radon-Nikod$\acute{y}$m property of $\mathbb{L}$-Banach spaces. Meanwhile, an example constructed in this paper shows that there do exist an $\mathbb{L}$-Banach space which fails to possess the Radon-Nikod$\acute{y}$m property. Finally, based on above work, we establish the dual representation theorem of $\mathbb{L}$-Bochner integrable function spaces, which extends and improves the corresponding classical result.
\end{abstract}

\maketitle

\section{Introduction}
The theory of $AW^{*}$-modules, introduced by Kaplansky \cite{ref3} and used to handle problems in $AW^{*}$-algebras slightly extending von Neumann algebras, generalized the theory of Hilbert spaces. An abelian $AW^{*}$-algebra is a $C(Z)$-space where $Z$ is a Stonean space, i.e., a compact Hausdorff space which is extremally disconnected. Kaplansky initiated the study of Kaplansky-Hilbert modules (briefly, $KH$-modules), in which an abelian $AW^{*}$-algebra $\mathbb{A}$ replaces the scalar field $\mathbb{C}$ of complex numbers. Besides, the stochastic theory in Riesz spaces has also been systematically and comprehensively developed, see \cite{ref5} and references therein for details. The central idea of such a theory is to replace the space of integrable functions $L^{1}(\Omega)$ with a vector lattice $\mathbb{E}$ with a weakly ordered identity which is Dedekind complete and allocated with an operator conditional expectation: $T: \mathbb{E}\rightarrow \mathbb{E}$ preserving the weakly ordered identity and ensuring the Dedekind completeness of $R(T).$ In fact, both $R(T)$ and $\mathbb{A}$ are Dedekind complete unital $f$-algebras. In order to unify such two theories and set up a general theory of functional analysis (i.e., $\mathbb{L}$-functional analysis), Kikianty et al. first proposed to replace the scalar field $\mathbb{R}$ (resp. $\mathbb{C}$) with a real (resp. complex) Dedekind complete unital $f$-algebra $\mathbb{L}$ and systematically established some basic and important theories of $\mathbb{L}$-normed spaces and $\mathbb{L}$-inner spaces in \cite{ref1}. For more detailed information about $f$-algebras and lattices, please refer to \cite{ref14,ref20,ref22}.

Moreover, random functional analysis, as another general extension of classical functional analysis, randomly generalized ordinary normed spaces, locally convex spaces, etc., and has been widely applied to conditional risk measures \cite{ref10}. Specially, the concept of a random normed module (an $RN$ module for short), which was first proposed by Guo \cite{ref6}, is a typical generalization of a normed space and plays an important role in random functional analysis. Such a theory aims to randomize the metric to the equivalence classes of non-negative $\mathbb{R}$-valued $\mathcal{F}$-measurable random variables. The main difficulty in the study of $RN$ modules is that the complex hierarchical structures of $RN$ modules make the traditional theory of conjugate spaces less applicable.
The significant breakthrough in the field of $RN$ modules was made by Guo, attributed to his establishment of the Hahn-Banach theorem for an almost everywhere (briefly, a.e.) bounded random linear functional \cite{ref6}. Furthermore, the representation theory of random conjugate spaces and characterizations of the random reflexivity within the framework of random conjugate spaces have yielded some profound results \cite{ref8,ref9}. For example, Guo \cite{ref7} deeply studied the dual representation theorems of Lebesgue-Bochner function spaces, extending fundamental results in the theory of Banach spaces. Subsequently, Zhang extensively studied the theory of specific operator semigroups and established a mean ergodic characterization of a complete $RN$ module in \cite{ref23,ref24,ref25}. Wu, Zeng and Zhao \cite{ref15} promoted the development of $L^{0}$-convex compactness in random locally convex modules, and Wu established the Farkas' lemma and the Bishop-Phelps theorem in fixed random spaces respectively \cite{ref16,ref17}.

The theory of the Radon-Nikod$\acute{y}$m property (RNP for short) of Banach spaces has been occupying a crucial position in the development of Banach spaces. Rieffel \cite{ref28} gave a new hypothesis of the Radon-Nikod$\acute{y}$m theorem for Bochner integrals that is equivalent to the classical case. Andrews \cite{ref26} established the conditions for the space of bounded linear operators to possess the RNP. Subsequently, Cullender and Labuschagne \cite{ref21} characterized Banach spaces with the RNP, extending a classical martingale characterization of the RNP. In this paper, we propose the RNP of $\mathbb{L}$-Banach spaces. This definition is indispensable because it can prove the main theorem of this paper, which is Theorem 3.4, and there do exist an $\mathbb{L}$-Banach space which fails to possess the RNP, namely Example 3.3, which is quite different from the classical case. In the process of proving the dual representation theorem of $\mathbb{L}$-Bochner function spaces, it is a remarkable fact that the value domain of functions in such spaces is an $\mathbb{L}$-Banach space, so the absolute value of functions cannot be simply taken, leading us to explore new methods different from the classical one, which is one of the difficulties of this paper.

Classical dual representation theorems, including the conjugate spaces of $L^{p}(\Omega,\mathcal{A},$ $\mu)(1\leq p\leq+\infty)$ and $C[0,1]$, are of vital significance to functional analysis, and such theorems all involve Lebesgue integrals, measure spaces and the RNP. This paper mainly discusses the dual representation theorem of an $\mathbb{L}$-valued analogue of the canonical $L^{p}(\Omega,\mathcal{A},\mu)$. The symbol $C_{\infty}(Z)$ represents the collection formed by all extended $\mathbb{R}$-valued continuous functions $\varphi$ on $Z$ that is Stonean satisfying that $\varphi$ takes finite values on a dense subset $U_{\varphi}$ of $Z.$ The connection of $\mathbb{L}$, $C(Z)$ and $C_{\infty}(Z)$ can be given as $C(Z)\subseteq\mathbb{L}\subseteq C_{\infty}(Z),$ which is called the representation theory. Many scholars have researched various integrals of functions taking values in abstract spaces and obtained deep developments \cite{ref18,ref27,ref29}. Integrals of Banach-valued functions that have been involved include McShane integrals, Bochner integrals and so on \cite{ref30}. Among them, McShane and Henstock integrals are similar to Riemann integrals, except for slight changes in the partition of the domain, which make handling ill-conditioned functions still difficult. Besides, Pettis and Dunford integrals involve the weak $\mu$-measurability. Motivated by the work of Guo in \cite{ref7,ref9}, we study Bochner integrals and find that the precondition for Banach-valued functions to be Bochner integrable is the strong $\mu$-measurability, which makes it more convenient to study  conjugate spaces. Based on the above considerations, we extend the definitions of $\mu$-measurable functions and Bochner integrable functions in \cite{ref11} to $\mathbb{L}$-functional analysis, and give the corresponding concepts of $\mathbb{L}$-$\mu$-measurable functions and $\mathbb{L}$-Bochner integrable functions on a finite measure space $(S,\mathcal{F},\mu).$ Moreover, we give an $\mathbb{L}$-valued analogue of the canonical $L^{p}(\Omega,\mathcal{F},\mu),$ i.e., $L^{p}(\mu,X),$ investigate the completeness of such a space and further generalize several classical results.

This paper is arranged as following: in Section 2, we briefly recall some important and essential concepts and properties and establish $\mathbb{L}$-Bochner integrals in $\mathbb{L}$-functional analysis; in Section 3, we first give an example to show that not all $\mathbb{L}$-Banach spaces possess the RNP, and then we devote to proving the dual representation theorem of $\mathbb{L}$-Bochner function spaces.

\section{Preliminaries}
Throughout this paper, $\mathbb{L}$ is a fixed (real or complex) Dedekind complete unital $f$-algebra and $\mathbb{L}^{+}=\{\xi\in \mathbb{L}~|~\xi\geq 0 \}$. For the reader's clearness, we will first recall some concepts and facts in $\mathbb{L}$-functional analysis.
\begin{definition}\cite{ref1}
An ordered pair $(X, \|\cdot\|)$ is called an $\mathbb{L}$-normed space if $X$ is an $\mathbb{L}$-vector space and $\|\cdot\|$ is a mapping from $X$ to $\mathbb{L}^{+}$ which satisfies three axioms:
\begin{enumerate}
\item $\|x\|=0$ iff $x=\textbf{0},$ $\forall x\in X,$ where 0 stands the null element in $\mathbb{L}^{+}$ and $\textbf{0}$ the null vector in $X;$

\item $\|\mu x\|=|\mu|\|x\|,$  $\forall\mu\in \mathbb{L}$ and $x\in X$;

\item $\|x+y\|\leq\|x\|+\|y\|,$ $\forall x,y\in X,$

\end{enumerate}
where $\|\cdot\|$ is called an $\mathbb{L}$-norm on $X.$
\end{definition}

In particular, $(\mathbb{L},|\cdot|)$ is an $\mathbb{L}$-normed space, where $|\cdot|$ is the modulus operation on $\mathbb{L}.$

Unlike the classical case, the convergence of elements in an $\mathbb{L}$-normed space is similar to the order convergence of nets in a vector lattice, see \cite{ref2, ref4, ref12} for details.

\begin{definition}\cite{ref1}
Suppose that $(X,\|\cdot\|)$ is an $\mathbb{L}$-normed space and $(x_{\alpha})_{\alpha\in I}$ a net in $X.$ If there is a set $\mathcal{E}\subseteq \mathbb{L}^{+}$ with $\inf\mathcal{E}=0$ (briefly, $\mathcal{E}\downarrow 0$, $\mathcal{E}$ does not need to be a directed set) satisfying that for each $\varepsilon\in\mathcal{E}$ there is an $\alpha_{0}\in I$ satisfying
$\|x_{\alpha}-x\|\leq \varepsilon$ for any $\alpha\geq\alpha_{0},$
then we say the net $(x_{\alpha})_{\alpha\in I}$ converges to $x\in X$ and denote it by $x_{\alpha}\rightarrow x$ or $\lim_{\alpha} x_{\alpha}=x.$

A set $A\subseteq X$ is said to be closed if we can deduce that $x\in A$ for any $x_{\alpha}$ in $A$ satisfying $x_{\alpha}\rightarrow x.$

A function $f:X\rightarrow \mathbb{L}$ is said to be continuous if we can infer that $(f\left(x_{\alpha}\right))_{\alpha\in I}$ converges to $f(x)$ from $(x_{\alpha})_{\alpha\in I}$ converging to $x.$

A net $(x_{\alpha})_{\alpha\in I}$ is said to be Cauchy if there is a set $\mathcal{E}\downarrow 0$ satisfying that for each $\varepsilon\in\mathcal{E}$ there is an $\alpha_{0}\in I$ satisfying
$
\|x_{\alpha}-x_{\beta}\|\leq \varepsilon
$
for any $\alpha,\beta\geq\alpha_{0}.$
If each Cauchy net in $X$ converges, then $X$ is said to be complete or called an $\mathbb{L}$-Banach space.

$X^{*},$ the dual space of $X,$ denotes $B(X,\mathbb{L})$, i.e., the space of all bounded linear operators $\phi:X\rightarrow\mathbb{L}.$ It is noteworthy that $X^{*}$ is an $\mathbb{L}$-Banach space for every $\mathbb{L}$-normed space $X.$
\end{definition}

Throughout the remainder of this paper, $(X, \|\cdot\|)$ always denotes an $\mathbb{L}$-Banach space, $N$ the set of positive integers, $S$ a nonempty set, $\mathcal{F}$ a $\sigma$-algebra on $S$ and $\mu$ a finite measure on the measurable space $(S,\mathcal{F}).$

In the classical case, assume that $(\Omega,\mathcal{A},\mu)$ is a measure space, then
\[
L^{p}(\Omega,\mathcal{A},\mu)=\left\{f \ is \ measurable \ on \ \Omega ~|~ \|f\|_{p}=\left(\int_{\Omega}|f|^{p}d\mu\right)^{\frac{1}{p}}<+\infty\right\}.
\]
Therefore, in order to define an $\mathbb{L}$-valued analogue of the above space in $\mathbb{L}$-functional analysis, it is essential and difficult to construct the concepts of $\mathbb{L}$-$\mu$-measurable functions and $\mathbb{L}$-Bochner integrable functions. The definitions presented below generalize the corresponding notions in \cite{ref11}.

\begin{definition}
For a given finite measure space $(S,\mathcal{F},\mu),$ $I_{F}$ represents the characteristic function of $F\in\mathcal{F}.$ In the following of this definition, we always assume that g is a function from $S$ to $X.$

$g$ is said to be $\mathbb{L}$-simple, if there are $x_{1},x_{2},\cdots x_{n}\in X$ and mutually disjoint $S_{1},S_{2},$ $\cdots S_{n}\in S$ satisfying $g=\sum\limits_{j=1}\limits^{n}x_{j}I_{S_{j}},$ where $n\in N$.

$g$ is said to be $\mathbb{L}$-strongly $\mu$-measurable (briefly, $\mathbb{L}$-$\mu$-measurable), if there is a sequence of $\mathbb{L}$-simple functions $(g_{n})_{n\in N}$ satisfying
$
\lim_{n}\left\|g_{n}-g\right\|=0,
$
$\mu$-a.e., where 0 stands the zero element of $\mathbb{L}.$

An $\mathbb{L}$-$\mu$-measurable function $g$ is said to be $\mathbb{L}$-Bochner integrable, if there is a sequence of $\mathbb{L}$-simple functions $(g_{n})_{n\in N}$ satisfying
$
\lim_{n}\int_{S}\|g_{n}$ $-g\|d\mu=0.
$
The $\mathbb{L}$-Bochner integral $\int_{S}gd\mu$ is defined by
$
\int_{S}gd\mu=\lim_{n}\int_{S}g_{n}d\mu,
$
where $\int_{S}g_{n}d\mu$ of $\mathbb{L}$-simple functions $g_{n}=\sum\limits_{j=1}\limits^{k^{(n)}}x_{j}^{(n)}I_{S_{j}^{(n)}}$ is defined by
$
\int_{S}g_{n}d\mu=\sum\limits_{j=1}\limits^{k^{(n)}}x_{j}^{(n)}\mu\left(S_{j}^{(n)}\right).
$
\end{definition}

The classical $H\ddot{o}lder$ inequality and $Minkowski$ inequality \cite{ref19} play an important role in functional analysis. The following versions of these two inequalities in $\mathbb{L}$-functional analysis are presented in \cite{ref1}.
\begin{proposition}[$H\ddot{o}lder$ inequality]

Suppose $1\leq p,q\leq+\infty$ are a pair of conjugate numbers (i.e.,$\frac{1}{p}+\frac{1}{q}=1,$ and $\frac{1}{\infty}=0$ is universally acknowledged), then for any $u,v:S\rightarrow X$, we get
\[
\int_{S}\|u(t)v(t)\|d\mu(t)\leq\left(\int_{S}\|u(t)\|^{p}d\mu(t)\right)^{\frac{1}{p}}\left(\int_{S}\|v(t)\|^{q}d\mu(t)\right)^{\frac{1}{q}}.
\]
\end{proposition}
\begin{proposition}[$Minkowski$ inequality]

Suppose $1\leq p<+\infty,$ then for any $u,v:S\rightarrow X$, we get
\[
\left(\int_{S}\|u(t)+v(t)\|^{p}d\mu(t)\right)^{\frac{1}{p}}\leq\left(\int_{S}\|u(t)\|^{p}d\mu(t)\right)^{\frac{1}{p}}+\left(\int_{S}\|v(t)\|^{p}d\mu(t)\right)^{\frac{1}{p}}.
\]
\end{proposition}

For $1\leq p<\infty,$ the symbol  $L^{p}(S,\mathcal{F},\mu,X)$ (briefly, $L^{p}(\mu,X)$) represents the collection of mappings $u:S\rightarrow X$ satisfying $u$ is $\mathbb{L}$-Bochner integrable such that $$\|u\|_{p}:=\left(\int_{S}\|u(t)\|^{p}d\mu(t)\right)^{\frac{1}{p}}$$ exists in $\mathbb{L}.$ Moreover, such mappings can be represented as
$$
\|u\|_{p}=\left(\sup_{E\in\mathcal{F}}\int_{E}\|u(t)\|^{p}d\mu(t)\right)^{\frac{1}{p}}
$$
as the net of integrals on finite subsets $\left(\int_{E}\|u(t)\|^{p}d\mu(t)\right)_{E\in\mathcal{F}}$  is increasing and bounded by $\|u\|_{p}^{p}.$

Clearly, $\|\cdot\|_{p}$ satisfies (1) and (2) in Definition 2.1, while (3) can be obtained from Proposition 2.5. Thus we deduce that $L^{p}(\mu,X)$ is an $\mathbb{L}$-normed space equipped with the $\mathbb{L}$-norm $\|\cdot\|_{p}$.

Specially, when $p=\infty,$ $L^{\infty}(S,\mathcal{F},\mu,X)$ (briefly, $L^{\infty}(\mu,X)$) denotes the collection of mappings $u:S\rightarrow X$ satisfying $u$ is $\mathbb{L}$-Bochner integrable satisfying
\[
\|u\|_{\infty}:=ess\sup\|u\|
\]
exists in $\mathbb{L}.$

Next, we will give Lebesgue's dominated convergence theorem in $\mathbb{L}$-functional analysis.
\begin{theorem}

For a given finite measure space $(S,\mathcal{F},\mu),$ assume that $(g_{n})_{n\in N}$ is a sequence of $\mathbb{L}$-Bochner integrable functions on $S$ whose value domain is an $\mathbb{L}$-Banach space, and
$
g_{n}\stackrel{a.e.}{\longrightarrow}g \ or \ g_{n}\stackrel{\mu}{\longrightarrow}g.
$
If there is an $\mathbb{L}^{+}$-valued $\mathbb{L}$-Bochner integrable function $h$ satisfying $\|g_{n}\|\leq h,$ $\mu$-a.e. for any $n\in N$ and there is an extended $\mathbb{R}$-valued Lebesgue integrable function $\varphi$ satisfying $h(t)\leq\varphi,$ a.e. for $t\in S,$ then $g$ is $\mathbb{L}$-Bochner integrable and
$
\lim_{n}\int_{S}g_{n}d\mu=\int_{S}gd\mu.
$
\end{theorem}
\begin{proof}
Since $(g_{n})_{n\in N}$ is a sequence of $\mathbb{L}$-Bochner integrable functions, it follows that $\|g_{n}\|\in \mathbb{L}^{+}$ and $\|g\|\in \mathbb{L}^{+}$ are $\mathbb{L}$-$\mu$-measurable functions. We can obtain that there are $l(n)~(depending~ on~ n)$ sequences of $\mathbb{L}$-simple functions $(\widetilde{g}^{l(n)}_{n})_{l(n)\in N}$ for each $n\in N$ satisfying
$
\lim_{l(n)}\left\|\widetilde{g}^{l(n)}_{n}-g_{n}\right\|=0,
$
$\mu$-a.e., and then,
$
\lim_{l(n)}\left|\|\widetilde{g}^{l(n)}_{n}\|-\left\|g_{n}\right\|\right|=0,
$
$\mu$-a.e.. Obviously, $\left|\|\widetilde{g}^{l(n)}_{n}\|\right|$ is a non-negative $\mathbb{R}$-valued measurable function, it follows that $\|g_{n}\|$ and $\|g\|$ are non-negative $\mathbb{R}$-valued measurable functions.

If $g_{n}\stackrel{a.e.}{\longrightarrow}g,$ then
$
\|g_{n}-g\|\stackrel{a.e.}{\longrightarrow}0.
$
Further, one can obtain that
$
\mathbb{L}\ni\|g_{n}\|\stackrel{a.e.}{\longrightarrow}\|g\|\in \mathbb{L}.
$
If there are $h$ and $\varphi$ satisfying the conditions of this theorem, then,
$
\|g_{n}(t)-g(t)\|\leq2h(t)\leq2\varphi,
$
a.e. for $t\in S.$ Applying the classical dominated theorem to $\|g_{n}(t)-g(t)\|,$ we claim
$
\lim_{n}\int_{S}\|g_{n}(t)-g(t)\|d\mu(t)=0.
$
Further, the definition of $\mathbb{L}$-Bochner integrals results in
$
\lim_{n}\int_{S}g_{n}d\mu=\int_{S}gd\mu.
$

If $g_{n}\stackrel{\mu}{\longrightarrow}g,$ then $\|g_{n}\|\stackrel{\mu}{\longrightarrow}\|g\|.$ For each $(g_{n_{k}})_{k\in N}\subseteq (g_{n})_{n\in N},$ there is a subsequence $(g_{n_{k}^{'}})_{k\in N}\subseteq (g_{n_{k}})_{k\in N}$ satisfying $\|g_{n_{k}^{'}}\|\stackrel{a.e.}{\longrightarrow}\|g\|.$ Similarly, we have
$
\lim_{n}\int_{S}g_{n_{k}^{'}}d\mu=\int_{S}gd\mu,
$
which suggests that
$
\lim_{n}\int_{S}g_{n}d\mu=\int_{S}gd\mu.
$
\end{proof}
\begin{lemma}
If $\left\{h_{n},n\in N\right\}\subseteq L^{p}(\mu,X)$ and $h\in L^{p}(\mu,X)$ satisfying
\[
\lim_{n\rightarrow\infty}\int_{S}\|h_{n}(t)-h(t)\|d\mu(t)=0,
\]
then
$h_{n}\stackrel{\mu}{\longrightarrow}h.$
\end{lemma}
\begin{proof}
For any $\gamma>0,$
$$
\begin{aligned}
\int_{S}\|h_{n}-h\|d\mu
&\geq\int_{[\|h_{n}-h\|\geq\gamma]}\|h_{n}-h\|d\mu\\
&\geq\gamma\mu([\|h_{n}-h\|\geq\gamma]).
\end{aligned}
$$
Hence $\mu([\|h_{n}-h\|\geq\gamma])\rightarrow 0$ as $n\rightarrow\infty,$ i.e., $h_{n}\stackrel{\mu}{\longrightarrow}h.$
\end{proof}
\begin{theorem}
$L^{p}(\mu,X)$ with $1\leq p<+\infty$ is an $\mathbb{L}$-Banach space.
\end{theorem}
\begin{proof}
For any Cauchy net $(u_{\alpha})_{\alpha\in I}\subseteq L^{p}(\mu,X),$ there is a set $\mathcal{E}\downarrow 0$ satisfying that for any $\varepsilon\in\mathcal{E}$ there is an $\alpha_{0}\in I$ satisfying
\begin{equation}
\alpha,\beta\geq\alpha_{0}\Rightarrow\|u_{\alpha}-u_{\beta}\|_{p}\leq\varepsilon.
\end{equation}
Then obviously
\begin{equation}
\|u_{\alpha}(t)-u_{\beta}(t)\|\leq\|u_{\alpha}-u_{\beta}\|_{p}
\end{equation}
holds for any $t\in S,$ it follows that $(u_{\alpha}(t))_{\alpha\in I}$ is a Cauchy net in $X$ for any $t\in S.$ Define $u:S\rightarrow X$ by
$
u(t)=\lim_{\alpha}u_{\alpha}(t),
$
which is  well-defined due to the completeness of $X.$ It is sufficient to certify that $u\in L^{p}(\mu,X)$ and $u_{\alpha}\rightarrow u$ in $L^{p}(\mu,X).$

According to Proposition 2.5, we have for any $\alpha,$
\begin{equation}
\begin{split}
&\left(\int_{S}\|u(t)\|^{p}d\mu(t)\right)^{\frac{1}{p}}\\
&\leq\left(\int_{S}\|u_{\alpha}(t)\|^{p}d\mu(t)\right)^{\frac{1}{p}}+\left(\int_{S}\|u(t)-u_{\alpha}(t)\|^{p}d\mu(t)\right)^{\frac{1}{p}}\\
&=\|u_{\alpha}\|_{p}+\left(\int_{S}\|u(t)-u_{\alpha}(t)\|^{p}d\mu(t)\right)^{\frac{1}{p}}.
\end{split}
\end{equation}
Choose $\varepsilon\in \mathcal{E}$ and $\alpha_{0}$ as in (2.1). Then, it follows from (2.2) that
\[
\alpha,\beta\geq\alpha_{0}\Rightarrow \int_{S}\|u_{\alpha}(t)-u_{\beta}(t)\|^{p}d\mu(t)\leq \varepsilon^{p}\mu(S).
\]
Since for each $t\in S,$ $\lim_{\beta}u_{\beta}(t)=u(t)$ in $X$, it is easy to get that
$\lim_{\beta}(u_{\alpha}(t)-u_{\beta}(t))$=$u_{\alpha}(t)-u(t)$
for every $\alpha\geq\alpha_{0}$ and $t\in S.$
Consequently, combining (2.1), (2.2) and Theorem 2.6, we have
$$
\begin{aligned}
\lim_{\beta}\int_{S}\|u_{\alpha}(t)-u_{\beta}(t)\|^{p}d\mu(t)&=\int_{S}\lim_{\beta}\|u_{\alpha}(t)-u_{\beta}(t)\|^{p}d\mu(t)\\
&=\int_{S}\|u_{\alpha}(t)-u(t)\|^{p}d\mu(t)
\end{aligned}
$$
for each $\alpha\geq\alpha_{0}.$ Then, since $\left\{\zeta\in \mathbb{L}: \zeta\leq \varepsilon^{p}\mu(S)\right\}$ is closed, it follows that
\begin{equation}
\left(\int_{S}\|u(t)-u_{\alpha}(t)\|^{p}d\mu(t)\right)^{\frac{1}{p}}\leq\varepsilon\mu(S)^{\frac{1}{p}}.
\end{equation}
Combining (2.3) and (2.4) yields $\|u\|_{p}<+\infty$  and (2.4) implies that $u_{\alpha}\rightarrow u$ as $n\rightarrow\infty$ in the $\|\cdot\|_{p}$-topology.

Next, we will check that $u$ is $\mathbb{L}$-Bochner integrable. Since $u_{\alpha}$ is $\mathbb{L}$-Bochner integrable for any $\alpha\in I,$ it follows that $u_{\alpha}$ is $\mathbb{L}$-$\mu$-measurable and there are $\alpha$ sequences of $\mathbb{L}$-simple functions $(u_{\alpha}^{n})_{n\in N}$ satisfying
$
\lim_{n}\int_{S}\|u_{\alpha}^{n}-u_{\alpha}\|d\mu=0.
$
According to Lemma 2.7, there are $\alpha$ subsequences $(u_{\alpha}^{n_{k}})_{k\in N}\subseteq(u_{\alpha}^{n})_{n\in N}$ such that
$
u_{\alpha}^{n_{k}}\stackrel{a.e.}{\longrightarrow} u_{\alpha},
$
which shows that
$
\lim_{k}\|u_{\alpha}^{n_{k}}-u_{\alpha}\|=0, a.e..
$

Since $u_{\alpha}$ is $\mathbb{L}$-$\mu$-measurable for any $\alpha\in I,$ there exist $\alpha$ sequences of $\mathbb{L}$-simple functions $(\bar{u}_{\alpha}^{n})_{n\in N}$ satisfying
$\lim_{n}\|\bar{u}_{\alpha}^{n}-u_{\alpha}\|=0,$ $\mu$-a.e.. From one perspective,
\begin{equation}
\lim_{n}\lim_{\alpha}\|\bar{u}_{\alpha}^{n}-u_{\alpha}\|=\lim_{\alpha}\lim_{n}\|\bar{u}_{\alpha}^{n}-u_{\alpha}\|=0, \mu-a.e..
\end{equation}
From another perspective, set
$
\bar{u}^{n}(t)=\lim_{\alpha}\bar{u}_{\alpha}^{n}(t)
$
for all $t\in S$ and $n\in N.$
Then, one has
\begin{equation}
\lim_{n}\lim_{\alpha}\|\bar{u}_{\alpha}^{n}-u_{\alpha}\|=\lim_{n}\|\bar{u}^{n}-u\|.
\end{equation}
Obviously, $(\bar{u}^{n})_{n\in N}$ is a sequence of $\mathbb{L}$-simple functions. Combining (2.5) and (2.6), we can obtain that
$\lim_{n}\|\bar{u}^{n}-u\|=0,$ $\mu$-a.e.,
that is, $u$ is $\mathbb{L}$-$\mu$-measurable.

Moreover, it is evident to obtain that
\begin{equation}
\lim_{n}\lim_{\alpha}\int_{S}\|u_{\alpha}^{n}-u_{\alpha}\|d\mu=\lim_{\alpha}\lim_{n}\int_{S}\|u_{\alpha}^{n}-u_{\alpha}\|d\mu=0.
\end{equation}
On the other side, set
$
u^{n}(t)=\lim_{\alpha}u_{\alpha}^{n}(t)
$
for all $t\in S$ and $n\in N,$ then there is an $\delta\in \mathbb{L}^{+}$ satisfying
$$
\begin{aligned}
\|u_{\alpha}^{n}(t)-u_{\alpha}(t)\|&\leq\|u_{\alpha}^{n}(t)-u^{n}(t)\|+\|u^{n}(t)-u_{\alpha}^{n_{k}}(t)\|+\|u_{\alpha}^{n_{k}}(t)-u_{\alpha}(t)\|\\
&\leq3\delta, a.e..
\end{aligned}
$$
Therefore, according to Theorem 2.6,
\begin{equation}
\lim_{n}\lim_{\alpha}\int_{S}\|u_{\alpha}^{n}-u_{\alpha}\|d\mu=\lim_{n}\int_{S}\lim_{\alpha}\|u_{\alpha}^{n}-u_{\alpha}\|d\mu=\lim_{n}\int_{S}\|u^{n}-u\|d\mu.
\end{equation}
Combining (2.7) and (2.8), we have
$
\lim_{n}\int_{S}\|u^{n}-u\|d\mu=0.
$
Consequently, $u$ is $\mathbb{L}$-Bochner integrable and this implies that $u\in L^{p}(\mu,X).$
\end{proof}

\section{conjugate spaces}
It is worth noting that in general, the order convergence of nets in $\mathbb{L}$-normed spaces is not topological, i.e., there is no topology on $\mathbb{L}$-normed spaces satisfying the topological convergence coincides with the order convergence of nets in such spaces. But this situation can happen if we put some restrictions on such spaces.
\begin{proposition}\cite{ref13}
Assume that $(Y,\|\cdot\|)$ is an $\mathbb{L}$-normed space, then the following items are equivalent:
\begin{enumerate}
\item There is a topology $\tau$ on $Y$ satisfying that for any net $(y_{\alpha})_{\alpha\in I}\subseteq Y,$ $y_{\alpha}\rightarrow y$ if and only if $y_{\alpha}\stackrel{\tau}{\longrightarrow}y;$
\item $\dim Y<+\infty.$

\end{enumerate}
\end{proposition}

A mapping $G:S\rightarrow X$ is called a measure on $X$ if it satisfies $G(\emptyset)=0$ and G is countably additive, i.e., $G(\bigcup_{n=1}^{\infty}F_{n})=\sum_{n=1}^{\infty}G(F_{n})$ in the $\mathbb{L}$-norm topology of $X$ for all sequences $(F_{n})_{n\in N}$ of mutually disjoint members of $\mathcal{F}.$

\begin{definition}
An $\mathbb{L}$-Banach space $X$ is said to possess the RNP with regard to $(S,\mathcal{F},\mu)$ if for any measure $G:\mathcal{F}\rightarrow X$ satisfying:
\begin{enumerate}
\item $G$ is $\mu$-continuous signified by $G\ll\mu$, i.e., for any $F\in\mathcal{F},$ $\mu(F)\rightarrow 0$ hints $G(F)\rightarrow 0.$

\item $G$ is of bounded variation, i.e.,
\[
|G|(S)=\sup_{\mathcal{N}}\sum\limits_{B\in \mathcal{N}}\|G(B)\|
\]
exists in $\mathbb{L}.$
where $\mathcal{N}$ is any partition of $S$ and divides $S$ into finite mutually disjoint elements of $\mathcal{F},$
there is $g\in L^{1}(\mu,X)$ satisfying
$G(F)=\int_{F}gd\mu$ for each $F\in\mathcal{F}.$

\end{enumerate}
Further, an $\mathbb{L}$-Banach space $X$ is said to possess the RNP if $X$ possesses this property with regard to every finite measure space.
\end{definition}
We give an example to show that this definition is indispensable because there do exist an $\mathbb{L}$-Banach space that does not possess this property.
\begin{example}
For a given finite measure space $(S,\mathcal{F},\mu),$ define a mapping $G:\mathcal{F}\rightarrow\mathbb{L}$ by
\[
G(F)=T(I_{F}), \forall F\in\mathcal{F}
\]
and
\[
T(u)=\sum\limits_{j}\frac{\mu(F_{j})\int_{F_{j}}u d\mu}{G(F_{j})}, \forall u\in L^{1}(\mu,\mathbb{L}),
\]
where $(F_{j})_{j\in N}$ is a sequence of mutually disjoint members of $\mathcal{F}$ and $\dim\mathbb{L}<\infty.$ Now we will prove that $\mathbb{L}$ does not possess the RNP.

In fact, $(\mathbb{L},|\cdot|)$ is a special $\mathbb{L}$-Banach space which has been proved in \cite{ref4}. Since $\dim\mathbb{L}<\infty,$ the fact that the order convergence of nets in $\mathbb{L}$ is topological holds according to Proposition 3.1. Thus, the cases involving compactness as well as other convergences in $\mathbb{L}$ are all true in the following proof, just as they are in the canonical case. It is evident to obtain that
$$
G(F)=T(I_{F})=\sum\limits_{j}\frac{\mu(F_{j})\int_{F_{j}}I_{F}d\mu}{G(F_{j})}=\sum\limits_{j}\frac{\mu(F_{j})\mu(F_{j}\bigcap F)}{G(F_{j})}.
$$
Consequently, if $\mu(F)\rightarrow 0,$ then $\mu(F_{j}\bigcap F)\rightarrow 0$ for any $j\in N,$ it is apparent that $G(F)\rightarrow 0,$ that is, $G\ll\mu.$

Obviously, $G$ is countably additive. Besides, one has for any $B\in \mathcal{N},$
\[
|G(B)|=|T(I_{B})|\leq\|T\|\|I_{B}\|
\]
exists in $\mathbb{L},$
where $\mathcal{N}$ is any finite partition of $S.$ Then
$
\sum\limits_{B\in \mathcal{N}}|G(B)|\leq\sum\limits_{B\in \mathcal{N}}\|T\|\|I_{B}\|,
$
and further
\[
|G|(S)=\sup_{\mathcal{N}}\sum\limits_{B\in \mathcal{N}}|G(B)|
\]
exists in $\mathbb{L},$ which indicates that $G$ is of bounded variation.
If $\mathbb{L}$ possesses the RNP with regard to $(S,\mathcal{F},\mu)$, then there is $v\in L^{1}(\mu,\mathbb{L})$ satisfying
$$
T(I_{F})=G(F)=\int_{F}vd\mu, \forall F\in\mathcal{F}.
$$
For each $\mathbb{L}$-simple function $h\in L^{1}(\mu,\mathbb{L})$, there exist $\varphi_{1},\varphi_{2},\cdots \varphi_{n}\in \mathbb{L}$ and mutually disjoint $S_{1},S_{2},\cdots S_{n}\in S$ satisfying $h=\sum\limits_{j=1}\limits^{n}\varphi_{j}I_{S_{j}},$ where $n\in N.$ one has
$$
T(h)=\sum\limits_{j=1}\limits^{n}\varphi_{j}T(I_{S_{j}})
=\sum\limits_{j=1}\limits^{n}\varphi_{j}\int_{S_{j}}vd\mu=\sum\limits_{j=1}\limits^{n}\varphi_{j}\int_{S}vI_{S_{j}}d\mu
=\int_{S}hvd\mu.
$$
Further, for each $u\in L^{1}(\mu,\mathbb{L}),$ there is a sequence of $\mathbb{L}$-simple functions $(u_{n})_{n\in N}\subseteq L^{1}(\mu,\mathbb{L})$ satisfying $u_{n}\stackrel{a.e.}{\longrightarrow}u.$ Thus
\[
T(u)=\int_{S}uvd\mu, \forall u\in L^{1}(\mu,\mathbb{L}).
\]

From one perspective, since $v$ is $\mathbb{L}$-Bochner integrable, there is a sequence of $\mathbb{L}$-simple functions $(v_{n})_{n\in N}$ satisfying
$
\lim_{n}\int_{S}\|v_{n}-v\|d\mu=0.
$
Next, for any $n\in N$, define the following operators family $T_{n}:L^{1}(\mu,\mathbb{L})\rightarrow \mathbb{L}$ by
\[
T_{n}(u)=\int_{S}uv_{n}d\mu, \forall u\in L^{1}(\mu,\mathbb{L}),
\]
then one can notice that $\dim R(T_{n})\leq\dim \mathbb{L}<+\infty,$ thus $T_{n}$ as continuous linear operators having finite ranks are compact operators for any $n\in N$. Moreover,
$$
\begin{aligned}
|\left(T_{n}-T\right)(u)|&=\left|\int_{S}u\left(v_{n}-v\right)d\mu\right|\\
&\leq\int_{S}\|u\|\|v_{n}-v\|d\mu\\
&\leq\|u\|_{1}\int_{S}\|v_{n}-v\|d\mu,
\end{aligned}
$$
which indicates that $T$ as the uniform limit of a compact operators family is compact. Consequently, the set
\[
M:=\left\{T(cI_{F}):F\in\mathcal{F} \ and \ c \ is \ a \ finite \ number\right\}
\]
is relatively compact.

From another perspective, without loss of generality, we assume that $v$ is not 0 everywhere. Take $\left(\frac{1}{v}I_{F_{n}}\right)_{n\in N}\subseteq M$ satisfying
\[
\mu(F_{n})=\frac{\mu(S)}{2} \ and \ \mu(F_{n}\bigtriangleup F_{m})=\frac{\mu(S)}{3}
\]
for $n\neq m.$ Since
$$
\begin{aligned}
\left|T\left(\frac{1}{v}I_{F_{n}}\right)-T\left(\frac{1}{v}I_{F_{m}}\right)\right|&=\left|\int_{S}\frac{1}{v}I_{F_{n}}vd\mu-\int_{S}\frac{1}{v}I_{F_{m}}vd\mu\right|\\
&\leq\int_{S}\left|I_{F_{n}}-I_{F_{m}}\right|d\mu\\
&=\mu(F_{n}\bigtriangleup F_{m})\\
&=\frac{\mu(S)}{3},
\end{aligned}
$$
$T\left(\frac{1}{v}I_{F_{n}}\right)$ does not have convergent subsequence, which contradicts with the compactness of $T.$
\end{example}

\begin{theorem} Let $(S,\mathcal{F},\mu)$ be a given finite measure space and $1\leq p<+\infty$, $1<q\leq+\infty$ a pair of $H\ddot{o}lder$ conjugate indices. If $(X^{*},\|\cdot\|_{X^{*}})$ possesses the RNP, then the conjugate space of $L^{p}(\mu,X),$ i.e., $\left(L^{p}(\mu,X)\right)^{*},$ is isometrically isomorphic with $L^{q}(\mu,X^{*})$ under the mapping
$$
F:v\in L^{q}(\mu,X^{*})\rightarrow F_{v}\in\left(L^{p}(\mu,X)\right)^{*},
$$
where for each $v\in L^{q}(\mu,X^{*}),$
$F_{v}$ $($denoting $F(v))$ $:L^{p}(\mu,X)\rightarrow \mathbb{L}$
is defined by
\[
F_{v}(u):=\int_{S}uvd\mu=\int_{S}\langle u,v\rangle d\mu=\int_{S}v(t)\left(u(t)\right)d\mu(t)
\]
for any $u\in L^{p}(\mu,X).$
\end{theorem}

For the sake of clearness, the proof is divided into the two Lemmas below. Lemma 3.5 claims that $F$ is well-defined and $\mathbb{L}$-norm preserving; Lemma 3.6 proves that $F$ is surjective by the RNP of $X^{*}.$

\begin{lemma}
$F$ is well-defined and isometric.
\end{lemma}
\begin{proof}
Let $v\in L^{q}(\mu,X^{*})$ and without loss of generality, we assume that $v$ is not 0 everywhere. According to the definition of the $\mathbb{L}$-Bochner integrability and Lemma 2.7, there is a sequence of $\mathbb{L}$-simple functions $(v_{n})_{n\in N}\subseteq L^{q}(\mu,X^{*})$ satisfying that there is a subsequence $(v_{n_{k}})_{k\in N}\subseteq(v_{n})_{n\in N}$ such that
$
v_{n_{k}}\stackrel{a.e.}{\longrightarrow} v.
$
Obviously, $\langle u,v_{n_{k}}\rangle$ is measurable for each $u\in L^{p}(\mu,X).$ Moreover, we can obtain without difficulty that
$
v_{n_{k}}(t)\stackrel{a.e.}{\longrightarrow}v(t)\in X^{*}
$
for any $t\in S,$ and then
$$
v_{n_{k}}(t)\left(u(t)\right)\stackrel{a.e.}{\longrightarrow}v(t)\left(u(t)\right)\in \mathbb{L},
$$
i.e.,
$
\lim_{k}\langle u,v_{n_{k}}\rangle=\langle u,v\rangle, a.e.,
$
which implies that $\langle u,v\rangle$ is measurable. Besides, since
$$
\left||\langle u,v_{n_{k}}\rangle|-|\langle u,v\rangle|\right|\leq|\langle u,v_{n_{k}}\rangle-\langle u,v\rangle|=\left|\langle u,v_{n_{k}}-v\rangle\right|,
$$
it follows that $|\langle u,v\rangle|$ is also measurable. According to Proposition 2.4, one can obtain that
\[
|F_{v}(u)|=\left|\int_{S}\langle u,v\rangle d\mu\right|\leq\int_{S}|\langle u,v\rangle| d\mu\leq\|u\|_{p}\|v\|_{q}.
\]
Therefore, $F$ is well-defined, $F_{v}\in\left(L^{p}(\mu,X)\right)^{*}$ and $\|F_{v}\|\leq\|v\|_{q}.$

As for
$
\|v\|_{q}\leq\|F_{v}\|,
$
we will divide it into $1<p<+\infty$ and $p=1$ two cases.

When $1<p<+\infty,$ let
$u_{1}(t)=\|v(t)\|_{X^{*}}^{\frac{1}{p}}\in \mathbb{L}$ for each $t\in S.$
According to the $\mathbb{L}$-Bochner integrability of $v$ and the continuity of $\|\cdot\|,$ $u_{1}$ is $\mathbb{L}$-$\mu$-measurable, then $u_{1}\in L^{p}(\mu,\mathbb{L}).$ Consequently,
$$
\begin{aligned}
\int_{S}\|v(t)\|_{X^{*}}^{1+\frac{1}{p}}d\mu(t)
&=\int_{S}\|v(t)\|_{X^{*}}\|v(t)\|_{X^{*}}^{\frac{1}{p}}d\mu(t)\\
&=F_{v}\left(\frac{\|v(t)\|_{X^{*}}\|v(t)\|_{X^{*}}^{\frac{1}{p}}}{v(t)}\right)\\
&\leq\|F_{v}\|\left\|\frac{\|v(t)\|_{X^{*}}\|v(t)\|_{X^{*}}^{\frac{1}{p}}}{v(t)}\right\|_{p}\\
&\leq\|F_{v}\|\left(\int_{S}\|v(t)\|_{X^{*}}d\mu(t)\right)^{\frac{1}{p}}\\
&=\|F_{v}\|\left(F_{v}\left(\frac{\|v(t)\|_{X^{*}}}{v(t)}\right)\right)^{\frac{1}{p}}\\
&\leq\|F_{v}\|^{1+\frac{1}{p}}\left(\mu(S)\right)^{\frac{1}{p^{2}}}.
\end{aligned}
$$
Let
$
u_{2}(t)=\|v(t)\|_{X^{*}}^{\frac{1}{p}+\frac{1}{p^{2}}}\in \mathbb{L}
$ for each $t\in S,$
then one can similarly obtain $u_{2}\in L^{p}(\mu,\mathbb{L}).$ Thus, we get
$$
\begin{aligned}
\int_{S}\|v(t)\|_{X^{*}}^{1+\frac{1}{p}+\frac{1}{p^{2}}}d\mu(t)
&=\int_{S}\|v(t)\|_{X^{*}}\|v(t)\|_{X^{*}}^{\frac{1}{p}+\frac{1}{p^{2}}}d\mu(t)\\
&=F_{v}\left(\frac{\|v(t)\|_{X^{*}}\|v(t)\|_{X^{*}}^{\frac{1}{p}+\frac{1}{p^{2}}}}{v(t)}\right)\\
&\leq\|F_{v}\|\left(\int_{S}\|v(t)\|_{X^{*}}^{1+\frac{1}{p}}d\mu(t)\right)^{\frac{1}{p}}\\
&=\|F_{v}\|^{1+\frac{1}{p}+\frac{1}{p^{2}}}\left(\mu(S)\right)^{\frac{1}{p^{3}}}.
\end{aligned}
$$
Analogously, let
$
u_{n}(t)=\|v(t)\|_{X^{*}}^{\frac{1}{p}+\frac{1}{p^{2}}+\cdots+\frac{1}{p^{n}}}\in \mathbb{L}
$ for each $t\in S.$
Then
\begin{equation}
\int_{S}\|v(t)\|_{X^{*}}^{1+\frac{1}{p}+\frac{1}{p^{2}}+\cdots+\frac{1}{p^{n}}}d\mu(t)\leq\|F_{v}\|^{1+\frac{1}{p}+\frac{1}{p^{2}}+\cdots+\frac{1}{p^{n}}}\left(\mu(S)\right)^{\frac{1}{p^{n+1}}}
\end{equation}
for any $n\in N.$ Observing that $\sum\limits_{n=0}\limits^{\infty}\frac{1}{p^{n}}=q$ when $p>1$ and $\|u_{n}\|\leq\|v\|^{q}\in \mathbb{L}.$ Consequently, it follows from Theorem 2.6 that
$
\|v\|_{q}\leq\|F_{v}\|.
$

When $p=1,$ set
$
E_{\varepsilon}=\left\{t\in S ~|~ \|v(t)\|_{X^{*}}>\|F_{v}\|+\varepsilon\right\}
$
for any positive number $\varepsilon.$
It follows from (3.1) that
$
\int_{S}\|v(t)\|_{X^{*}}^{n+1}d\mu(t)\leq\|F_{v}\|^{n+1}\mu(S).
$
Consequently,
$
\mu\left(E_{\varepsilon}\right)\left(\|F_{v}\|+\varepsilon\right)^{n+1}\leq\|F_{v}\|^{n+1}\mu(S),
$
and further we can obtain that
\[
\left(\frac{\mu\left(E_{\varepsilon}\right)}{\mu(S)}\right)^{\frac{1}{n+1}}\leq\frac{\|F_{v}\|}{\|F_{v}\|+\varepsilon}<1.
\]
Taking $n\rightarrow\infty$ in the above inequality, we have $\mu\left(E_{\varepsilon}\right)=0,$ which indicates that $\|v\|_{\infty}\leq\|F_{v}\|.$
\end{proof}
\begin{lemma}
$F$ is surjective, namely, for any $H\in \left(L^{p}(\mu,X)\right)^{*},$ there is $v\in L^{q}(\mu,X^{*})$ satisfying $F_{v}=H.$
\end{lemma}
\begin{proof}
Define a mapping $G:S\rightarrow X^{*}$ by
$
G(F)=H(I_{F}), \forall F\in \mathcal{F}.
$

It is simple to verify that $G(\emptyset)=0$ and $G$ is a countably additive measure.

For any $(F_{m})_{m\in N}\subseteq\mathcal{F}$ satisfying $F_{m}\downarrow\emptyset,$
$$
\|G(F_{m})\|_{X^{*}}=\sup_{\|x\|\leq 1}|G(F_{m})(x)|=\sup_{\|x\|\leq 1}|H(xI_{F_{m}})|\leq\|H\|\|I_{F_{m}}\|_{p},
$$
thus
$
\lim_{m}G(F_{m})=0,
$
which suggests that $G$ is $\mu$-continuous.

Moreover, for each $\mathcal{N}=\left\{B_{j} ~|~ j=1,2,\ldots n, n\in N\right\}$ which is any partition of $S$ and divides $S$ into finite mutually disjoint elements of $\mathcal{F}$, one has
$$
\begin{aligned}
\sum_{j=1}^{n}\|G(B_{j})\|_{X^{*}}
&=\sum_{j=1}^{n}\sup_{\|x_{j}\|\leq 1}|G(B_{j})(x_{j})|\\
&=\sum_{j=1}^{n}\sup_{\|x_{j}\|\leq 1}|H(x_{j}I_{B_{j}})|\\
&=\sup_{\|x_{j}\|\leq 1}H\left(\sum_{j=1}^{n}x_{j}I_{B_{j}}sgn\left(H(x_{j}I_{B_{j}})\right)\right)\\
&\leq\|H\|\left(\mu(S)\right)^{\frac{1}{p}}.
\end{aligned}
$$
Thus,
$
|G|(S)=\sup\limits_{\mathcal{N}}\sum\limits_{B_{j}\in\mathcal{N}}\|G(B_{j})\|_{X^{*}}\leq\|H\|\left(\mu(S)\right)^{\frac{1}{p}}
$
exists in $\mathbb{L},$
which indicates that $G$ is of bounded variation.

Since $X^{*}$ possesses the RNP, it follows that there is $v\in L^{1}(\mu,X^{*})$ satisfying
$G(F)=\int_{F}vd\mu$ for any $F\in\mathcal{F}.
$
Consequently,
\[
G(F)(x)=\int_{F}v(t)(x)d\mu(t)=H(xI_{F}), \forall F\in\mathcal{F} \ and \ x\in X.
\]

For each $\mathbb{L}$-simple function $u\in L^{p}(\mu,X)$, there exist $x_{1},x_{2},\cdots x_{n}\in X$ and mutually disjoint $S_{1},S_{2},\cdots S_{n}\in S$ satisfying $u=\sum\limits_{j=1}\limits^{n}x_{j}I_{S_{j}},$ where $n\in N.$ Thus
$$
H(u)=\sum\limits_{j=1}\limits^{n}\int_{S}v(t)\left(x_{j}I_{S_{j}}(t)\right)d\mu(t)
=\int_{S}v(t)\left(\sum\limits_{j=1}\limits^{n}x_{j}I_{S_{j}}(t)\right)d\mu(t)
=\int_{S}\langle u,v\rangle d\mu.
$$

For any $u\in L^{p}(\mu,X),$ there is a sequence of $\mathbb{L}$-simple functions $(u_{n})_{n\in N}\subseteq L^{p}(\mu,X)$ satisfying
$u_{n}$ converges to $u,$ a.e..
Then clearly
$
\langle u_{n},v\rangle\stackrel{a.e.}{\longrightarrow}\langle u,v\rangle\in \mathbb{L}.
$
Moreover, $v\in L^{q}(\mu,X^{*})$ can be proved similarly to Lemma 3.5. Besides, according to Proposition 2.4,
$$
\begin{aligned}
\left|\int_{S}\langle u_{n},v\rangle-\langle u,v\rangle d\mu\right|&\leq\int_{S}|\langle u_{n},v\rangle-\langle u,v\rangle|d\mu\\
&=\int_{S}|\langle u_{n}-u,v\rangle|d\mu\\
&\leq\|u_{n}-u\|_{p}\|v\|_{q},
\end{aligned}
$$
which implies that
$
\lim_{n}\int_{S}\langle u_{n},v\rangle d\mu=\int_{S}\langle u,v\rangle d\mu.
$
Consequently, one can immediately obtain
\[
H(u)=\lim_{n\rightarrow\infty}H(u_{n})=\lim_{n\rightarrow\infty}\int_{S}\langle u_{n},v\rangle d\mu=\int_{S}\langle u,v\rangle d\mu
\]
for each $u\in L^{p}(\mu,X),$  which suggests that $F_{v}=H.$
\end{proof}
\begin{remark}
Particularly, if we take $\mathbb{L}=\mathbb{R}$ in Theorem 3.4, then $X$ is exactly an ordinary Banach space and for the conjugate space of $L^{p}(\mu,X)$ as the space of Lebesgue-Bochner functions, we have $\left(L^{p}(\mu,X)\right)^{*}=L^{q}(\mu,X^{*}).$

Besides, let $(X,\|\cdot\|)=(\mathbb{R},|\cdot|),$ then Theorem 3.4 generalizes the corresponding classical result since $L^{p}(\mu,X)$ reduces to the space of $\mathbb{R}$-valued Lebesgue-measurable functions.
\end{remark}

\bibliographystyle{amsplain}

\end{document}